\documentclass[reqno,11pt,twoside,english]{amsart}
\usepackage[T1]{fontenc}
\usepackage[latin9]{inputenc}
\usepackage{amsthm}
\usepackage{amssymb}
\usepackage{amsmath,amsfonts,epsfig}

\makeatletter

\numberwithin{equation}{section}

\usepackage{color}
\usepackage[final,linkcolor = blue,citecolor = blue,colorlinks=true]{hyperref}

 \newcommand{\rr}{{\mathbb R}}
 \newcommand{\R}{{\mathbb R}}
  \newcommand {\F}{{\mathbb F}}
 \newcommand\rn{{{\rr}^n}}

 \newcommand\fz{\infty}
  \newcommand\wq{\infty}
 \newcommand\az{\alpha}

 \newcommand\loc{{\mathop\mathrm{\,loc\,}}}
 \newcommand\lip{{\mathop\mathrm{\,Lip}}}

 \newcommand\gz{{\gamma}}
 \newcommand\ga{{\gamma}}
 
 \newcommand\boz{{\Omega}}
  \newcommand\Om{{\Omega}}

 \newcommand\wz{\widetilde}

 \newcommand\esup{\mathop{\mathrm{ess\,sup\,}}}

 \newcommand*{\bR}{\ensuremath{\mathbb{R}}}
 \newcommand*{\ssub}{\subset\subset}
\DeclareMathOperator{\Lip}{Lip}

\numberwithin{equation}{section}
\newtheorem{theorem}{Theorem}[section]

\newtheorem{lemma}[theorem]{Lemma}

\newtheorem{proposition}[theorem]{Proposition}

\theoremstyle{definition}
\begingroup
\newtheorem{definition}[theorem]{Definition}

\endgroup

\makeatother

\usepackage{babel}
\begin{document}

\title[$L^\infty$-variational problems associated to Finsler structures]{$L^\infty$-variational problems associated to measurable Finsler structures}

\author{Chang-Yu Guo,  Chang-lin Xiang and Dachun Yang}

\address[Chang-Yu Guo]{Department of Mathematics and Statistics, University of Jyv\"askyl\"a, P.O. Box 35, FI-40014, Jyv\"askyl\"a, Finland}
\email{changyu.c.guo@jyu.fi}

\address[Chang-Lin Xiang]{Department of Mathematics and Statistics, University of Jyv\"askyl\"a, P.O. Box 35, FI-40014, Jyv\"askyl\"a, Finland}
\email{changlin.c.xiang@jyu.fi}

\address[Dachun Yang]{School of Mathematical Sciences, Beijing Normal University, Laboratory of Mathematics and Complex Systems, Ministry of Education, Beijing 100875, People's Republic of China}
\email{dcyang@bnu.edu.cn}

\thanks{C.Y. Guo was partially supported by the Academy of Finland grant 131477 and the Magnus Ehrnrooth Foundation.  C.-L.  Xiang was supported by the Academy of Finland grant 259224.  D. Yang was supported by the National Natural Science Foundation of China (Grant Nos. 11171027 and 11361020), the Specialized Research Fund for the Doctoral Program of Higher Education of China (Grant No.  20120003110003) and the Fundamental Research Funds for Central Universities of China (Grant Nos. 2013YB60 and 2014KJJCA10).}

\begin{abstract}
We study $L^{\infty}$-variational problems associated to measurable Finsler structures  in Euclidean spaces. We obtain existence and uniqueness results for the absolute minimizers.
\end{abstract}

\maketitle

{\small
\smallskip
\newline
\tableofcontents}
\bigskip

\arraycolsep=1pt

\section{Introduction}
In this paper, we study the $L^\wq$-variational problem 
\begin{equation}\label{eq: infinite variational prob.}
 \F(u;U):=\underset{x\in U}{\esup} F(x,\nabla u(x))
\end{equation}
over the class of Lipschitz functions on $U\ssub\boz$ with a given boundary
data, where $U\subset\subset\Om$ is an arbitrary open subset of a given domain $\Om$ in the Euclidean space $\R^n$, $n\ge 2$, and $F:\Om\times\R^n\to \R$ is a Borel measurable Finsler structure on $\Om$ (see Definition \ref{def:finsler structure} below). Above, $\nabla u(x)$ denotes the gradient of $u$ at $x$. By Rademacher's theorem, any locally Lipschitz continuous function is differentiable at almost every point, and hence~\eqref{eq: infinite variational prob.} makes sense. For applications of this $L^\infty$ calculus of variation,  see~\cite{hil05,gnp01} and the references therein.

The study of $L^\wq$-variational problems of type (\ref{eq: infinite variational prob.}) 
was initiated by Aronsson \cite{a1,a2,a3,a4} in the model case
$F(x,\,\xi):=|\xi|^2$.  That is,  consider  the functional~\eqref{eq: infinite variational prob.} in the form
\begin{equation}\label{eq:infinite harmonic}
\F(u,U):= \underset{x\in U}{\esup}|\nabla u(x)|^2.
\end{equation} 
Since then  the study of the $L^\fz$-variational problem,
for more general functionals $F$ with various smoothness assumptions,
has advanced significantly; see the seminal works \cite{ceg,j93} and the survey paper \cite{acj} for more information on the recent developments.
The $L^\fz$-variational problem is also interesting even if the functional $F$
is not smooth or even not continuous;
see for example \cite{acj,cp,cpp,gpp,ksz-2014} and the references
therein. In the following, we  first briefly review some results on  the model case ~\eqref{eq:infinite   harmonic}. This model case is of great importance,  due to its simple structure, and due to  all the techniques that are developed to study the existence and the uniqueness of~\eqref{eq:infinite harmonic} can be possibly applied to the general functional of the form~\eqref{eq: infinite variational prob.}. Then we present the main result of this paper.


In the model case ~\eqref{eq:infinite
   harmonic}, Aronsson introduced the idea of  absolute minimizers  in his series of
papers  \cite{a1,a2,a3,a4}. His idea easily extends to the general
case ~\eqref{eq: infinite variational prob.}.   Precisely,
 let $U\subset\subset  \Om$ be an arbitrary open subset. Denote by $\lip(U)$ the  space of Lipschitz continuous functions on $U$ with respect to the standard Euclidean metric, and by $\lip_{\loc}(U)$
the space of locally Lipschitz continuous functions on $U$.  A function $u\in \lip_{\loc}(U)\cap C(\overline U)$ is called an
{\it absolute  minimizer  for $F$} on $U$ if for every open subset
$V\subset\subset U$  and $v\in \lip_{\loc}(V)\cap C(\overline V)$ with
  $u|_{\partial V}= v|_{\partial V}$, we have $\F(u,V)\le \F(v,V)$, that is,
$$ \underset{x\in
V}{\esup} F(x, \nabla u(x)) \le  \underset{x\in
V}{\esup} F(x, \nabla v(x)).$$ Moreover, given a
function $f\in \lip (\partial U)$,  $u\in \lip_{\loc}(U)\cap
C(\overline U)$ is called an
 {\it absolutely minimizing Lipschitz extension } of $f$ on $U$ with respect to $F$ if $u$ is an
absolute  minimizer  for $F$ on $U$ and $u|_{\partial U}=f$. In
literature, an absolute minimizer of the model case
~\eqref{eq:infinite harmonic} is also termed as an \textit{infinity
harmonic function} in $U$.

Aronsson~\cite{a3} proved the existence of absolute minimizers
 for ~\eqref{eq:infinite
   harmonic} with given Lipschitz Dirichlet boundary data. His approach is as
follows:  for a given bounded domain $\Omega\subset\bR^n$ and
$g\in \Lip(\partial \Omega)$, find the ``best" Lipschitz extension
of $g$ to $\Omega$. By the best extension, we mean that the
function should satisfy the condition
\begin{equation}\label{eq:best extension}
    L_u(V)=L_u(\partial V)\quad \text{ for all } V\subset \Omega,
\end{equation}
where $L_u(E):=\sup_{x,y\in E}\frac{u(y)-u(x)}{|y-x|}$ denotes the
smallest Lipschitz constant of $u$ in a set $E$. A brief review of the motivation of this approach can be found in Jensen's seminal work~\cite{j93}. Notice that for
any function $v$ in any domain $V$ we have $L_v(V)\geq
L_v(\partial V)$. It is not a trivial work to attain the equality in \eqref{eq:best extension}. We can easily find the following Lipschitz extensions of
$g$ by the McShane(-Whitney) extension
\begin{align*}
\Psi(x)&:=\inf_{y\in \partial\Omega}\Big(g(y)+L_g(\partial \Omega)|y-x|\Big)\\
\Phi(x)&:=\inf_{y\in \partial\Omega}\Big(g(y)-L_g(\partial
\Omega)|y-x|\Big).
\end{align*}
But $\Psi$ and $\Phi$  do not satisfy~\eqref{eq:best extension}
except $\Psi\equiv \Phi$; see \cite{a3}. It turns out that the functions
which satisfy \eqref{eq:best extension} are exactly the ones that
are infinity harmonic; see \cite{acj}.

Aronsson~\cite{a3} also formally derived the infinity Laplace
equation
\begin{equation}\label{eq:infinitey Laplace equation}
    \Delta_\infty u(x):=\big(D^2u(x)\nabla u(x)\big)\cdot \nabla u(x)=0,
\end{equation}
as the Euler-Lagrange equation of the variational problem ~\eqref{eq:infinite
   harmonic}, where $D^2u(x)$ denotes the Hessian of $u$ at $x$. Aronsson proved that a $C^2$-solution $u$ is
infinity harmonic if and only if it satisfies \eqref{eq:infinitey Laplace equation}. Of course
 at that time, he did not have the right tools to interpret the equation
  \eqref{eq:infinitey Laplace equation} for non-smooth functions. This was
   a major problem since there are non-smooth infinity harmonic functions,
    such as $u(x,y)=y^{4/3}-x^{4/3}$ in the plane.

After the development of the viscosity solution theory
 by Crandall and Lions in the 1980's, Jensen~\cite{j93} proved  that  infinity harmonic functions
are  viscosity solutions to equation~\eqref{eq:infinitey Laplace equation} and vice vesa,  under given Dirichlet boundary data. He also proved the existence and the uniqueness
results of infinity harmonic functions under more general Dirichlet
boundary data. As already remarked by Jensen~\cite{j93},  his  uniqueness 
approach uses equation \eqref{eq:infinitey Laplace equation} intensively. 
Thus, it seems hard to extend his  uniqueness approach to more general cases in which one can not derive an equation of type  \eqref{eq:infinitey Laplace equation} from the $L^{\infty}$ variational problem.

 More recent proofs for the uniqueness of Jensen~\cite{j93} can be found in 
Crandall, Gunnarsson and Wang~\cite{cgw07}, Barles and Busca~\cite{bb01}, and 
Armstrong and Smart~\cite{as}. Among these proofs, the key idea, to derive 
the uniqueness result for ~\eqref{eq:infinite harmonic}, is to use the
characterization of infinity harmonic functions via comparison
with cones, which was first properly stated by Crandall, Evans and
Gariepy \cite{ceg}. To gain some intuition, observe that for all
$a>0$, the cone function $C(x):=a|x|$ is a smooth solution of the
infinity Laplace equation \eqref{eq:infinitey Laplace equation} in
$\bR^n\backslash \{0\}$. This can be easily seen by noticing that
\begin{equation}\label{eq:cone function via gradient}
    |\nabla C(x)|=a\quad \text{ for all } x\neq 0.
\end{equation}
Differentiating \eqref{eq:cone function via gradient}, one easily obtains $\Delta_\infty C(x)=0$ in 
$\bR^n\backslash \{0\}$. In this regard, the cone function is a sort of fundamental solutions to the infinity Laplace equation and the comparison with cones is a sort of (weak) comparison principles. In~\cite{ceg}, a very elegant proof is used to show the equivalence of being infinity harmonic and satisfying the comparison with cones. We would like to point out that the comparison with cones has turned out to be a fruitful point of view, and for example, Savin's proof \cite{s05} for $C^1$ regularity of infinity harmonic functions in the plane is entirely based on cones; see also~\cite{es,wy}.

Our main aim of this paper is to consider the existence and the uniqueness results for the minimization
 problem~\eqref{eq: infinite variational prob.} associated with a very general Finsler structure $F$.
  The typical feature is that we impose very less regularity on $F$. In particular, in our case,
  there is no PDE associated to the variational problem~\eqref{eq: infinite variational prob.}
  and hence standard techniques from elliptic PDEs are not
  available. Our main result reads as follows.

\begin{theorem}\label{thm:main intro}
    Let $F:\Om\times\R^n\to \R$ be an admissible Finsler structure on $\Om$. Then for each open subset $U\ssub \Om$ and each boundary data
    $f\in\lip(\partial U)$, there exists a unique absolutely minimizing
    Lipschitz extension on $U$ with respect to $F$.
\end{theorem}

When $F(x,v):=\langle A(x)v, v\rangle$, where $A$ is a diffusion matrix-valued function, Theorem~\ref{thm:main intro} reduces to Theorem 5 in \cite{ksz-2014}. Thus, our main result can be regarded as a natural generalization of \cite[Theorem 5]{ksz-2014} from the diffusion case to more general Finsler case.

Although the general principle behind the proof of Theorem~\ref{thm:main intro} is similar to \cite[Theorem 5]{ksz-2014}, our approach (for the existence) is substantially different from \cite{ksz-2014}. Indeed, the proof given in \cite{ksz-2014} depends heavily on the speciality of the structure $F(x,v):=\langle A(x)v, v\rangle$ and seems not to be easily generalisable to our case.

For the existence result, our proof relies on the (crucial) Lemma
\ref{l4.5} and Proposition \ref{prop:essential differential
coincide with distance}, which allows us to describe the absolute
minimizer via the pointwise Lipschitz constant. In this step, we
also borrow some ideas from the recent related work \cite{g14},
which allow us to relate the geometric and the analytic aspects of
admissible Finsler structures.

For the uniqueness result, we follow closely the idea of \cite{as} and \cite{ksz-2014}, that is, we first characterize absolute minimizers for the variational problem \eqref{eq: infinite variational prob.} via comparison with cones, very similar to the infinity Laplace case. Then we establish the comparison with cones  as  in \cite{as}. Somewhat surprisingly, we do not really need equations (as in the infinity Laplace case) to effectively use the comparison with cones.

We remark here that there are many natural questions that can be done after this work. First of all, one could consider the linear approximation property for admissible Finsler structures with extra smoothness assumption as in \cite[Section 6]{ksz-2014} and the regularity issues as in \cite{es12} and \cite{swz14}. The second possible direction is to generalize these results to certain metric measure spaces as in \cite{jn,kz,ksz}.

This paper is organized as follows. In Section 2, we recall some
preliminary results on admissible Finsler sructures and the
associated (intrinsic) distances. In Section 3, we prove one of
the key results, namely, Proposition \ref{prop:essential
differential coincide with distance}. The proof of Theorem
\ref{thm:main intro} is contained in Section 4, as a special case
of the more general Theorem~\ref{thm:  Existence and Uniqueness}.
An alternative proof of Lemma \ref{lemma:key lemma} is provided in
the appendix.


Throughout the paper, we use $|\cdot|$ and $\langle\cdot,\cdot\rangle$ to denote the standard  norm and  inner product of Euclidean spaces.

\section{Preliminaries}
\subsection{Finsler structure and its dual}

Let $\Omega\subset\bR^n$, $n\geq 2$, be a domain. A Finsler structure on $\Om$ is defined as follows.

\begin{definition}[Finsler structure]\label{def:finsler structure}
We say that a function $F:\Omega\times \bR^n\to [0,\infty)$ is a Finsler structure on $\Omega$ if
\begin{itemize}
\item $F(\cdot,v)$ is Borel measurable for all $v\in \bR^n$, $F(x,\cdot)$ is continuous for a.e. $x\in \Omega$;
\item $F(x,v)>0$ for a.e. $x$ if $v\neq 0$;
\item $F(x,\lambda v)=|\lambda| F(x,v)$ for all $\lambda\in \bR$ and $(x,v)\in \Omega\times \bR^n$.
\end{itemize}
\end{definition}
It turns out to be too general for us to study   the $L^\wq$-variational problem (\ref{eq: infinite variational prob.}) in the context of Finsler structures. We will restrict ourselves to the class of  admissible Finsler structures.
\begin{definition}[Admissible Finsler structure]\label{def:admissible finsler stru}
A Finlser structure $F$ on $\Om$ is said to be admissible if
\begin{itemize}
\item $F(x,\cdot)$ is convex for a.e. $x\in \Omega$;
\item $F$ is locally equivalent to the Euclidean norm. That is, there exists a continuous function $\lambda:\Omega\to [1,\infty)$ such that
\begin{align*}
{\lambda(x)}^{-1}|v|\leq F(x,v)\leq \lambda(x)|v|.
\end{align*}

\end{itemize}
\end{definition}

For any admissible Finsler structure  $F$ on $\Omega$,  we introduce the dual $F^*:\Omega\times\bR^n\to [0,\infty)$ of $F$ in the following standard way.

\begin{definition}[Dual Finsler structure]\label{def:dual finsler} Let $F$ be an admissible Finsler structure on $\Omega$. We define  $F^*:\Omega\times\bR^n\to [0,\infty)$, the dual Finsler structure of  $F$ on $\Om$, by
\[
F^*(x,w):=\sup_{v\in \bR^n}\Big\{\langle v,w\rangle: F(x,v)\leq 1\Big\}.
\]
\end{definition}
We remark that it is direct to verify that
\begin{equation}\label{eq: dual Finsler structure formula}
  F^*(x,w)=\max_{v\in \R^n\backslash \{ 0\}}\left\langle w,\frac{v}{F(x,v)}\right\rangle .
\end{equation}
We shall need the following result on the properties of dual Finsler structures,  which can be found in~\cite[Section 1.2]{gpp06}.

\begin{proposition}[Basic properties of dual Finsler structures]\label{prop:basic prop of dual} Let $F$ be an admissible Finsler structure on $\Omega$. Then the dual Finsler structure $F^*$ has the following properties:
\begin{itemize}
\item $F^*(\cdot,v)$ is Borel measurable for all $v\in\bR^n$ and
$F^*(x,\cdot)$ is Lipschitz continuous for a.e. $x\in\Om$; \item
$F^*(x,\cdot)$ is a norm for a.e. $x\in \Om$;
\item Let  $\lambda(\cdot)$ be defined as in Definition \ref{def:admissible finsler stru}. Then $F^*(x,\cdot)$ satisfies that
\begin{align*}
{\lambda(x)}^{-1}|v|\leq F^*(x,v)\leq \lambda(x)|v|;
\end{align*}
\item $(F^*)^*(x,v)= F(x,v)$  for all $(x,v)\in \Omega\times\bR^n$.
\end{itemize}
\end{proposition}

\subsection{Intrinsic distance associated to an admissible Finsler structure}
For any admissible Finsler structure $F$ on $\Omega$, we  associate $\Omega$ with an intrinsic distance by setting
\begin{align*}
d_c^F(x,y):=\sup_{N}\inf_{\gamma\in \Gamma_N^{x,y}}\left\{\int_0^1 F(\gamma(t),\gamma'(t))dt\right\}\quad \text{ for all } x,y\in \Omega,
\end{align*}
where the supremum is taken over all subsets $N$ of $\Omega$ such that $|N|=0$ and $\Gamma_N^{x,y}(\Omega)$ denotes the set of all Lipschitz continuous curves $\ga$ in $\Omega$ with end points $x$ and $y$  such that $\mathcal{H}^1(N\cap \gamma)=0$ with $\mathcal{H}^1$ being the one dimensional Hausdorff measure. Similarly, we define the distance $d^{F^*}_c$ by
 \begin{align*}
d_c^{F^*}(x,y):=\sup_{N}\inf_{\gamma\in \Gamma_N^{x,y}}\left\{\int_0^1 F^*(\gamma(t),\gamma'(t))dt\right\} \quad \text{ for all } x,y\in \Omega.
\end{align*}
For notational simplicity, we write $d_c^*=d_c^{F^*}$ below.
We also need the following intrinsic distance function $\delta_{F}$ defined as
\[
\delta_{F}(x,y):=\sup\left\{u(x)-u(y):u\in\lip(\Omega),\|F(x,\nabla u)\|_{L^{\infty}(\Om)}\le 1\right\}.
\]
For the definition of the class of intrinsic distances, we refer to~\cite[Section 3]{dp95} or~\cite[Section 1.1]{gpp06}, and it will not play any role in this paper.

The following lemma implies that $\delta_{F}$ is  actually the same as $d_c^*$ at infinitesimal scale. The proof can be found in~\cite[Theorem 3.7]{dp93}. In the special case when $F$ is (weak) upper semicontinuous, a simpler proof can be found in~\cite[Proposition 3.1]{g14}.

\begin{lemma}\label{lemma:key lemma}
Let $F$ be an admissible Finsler structure on $\Omega$. Then for a.e. $x\in \Omega$, we have
\begin{align*}
\lim_{y\to x}\frac{\delta_F(x,y)}{d_c^*(x,y)}=1.
\end{align*}
\end{lemma}

\subsection{Comparison of metric derivatives}
For any distance $d$ on $\Omega$ and any Lipschitz continuous (with respect to $d$) curve $\gamma:[a,b]\to \Omega$, the length of $\gamma$ with respect to $d$ is denoted by $\mathcal{L}_d(\gamma)$, that is,
\begin{align*}
\mathcal{L}_d(\gamma):=\sup\left\{\sum_{i=1}^k d(\gamma(t_i),\gamma(t_{i+1})) \right\},
\end{align*}
where the supremum is taken over all partitions $\{[t_i,t_{i+1}]\}$ of $[a,b]$.

For any intrinsic distance $d$, which is locally equivalent to the Euclidean distance, we define
\begin{align}\label{def:new finsler distance}
\Delta_d(x,v):=\limsup_{t\to 0}\frac{d(x,x+tv)}{t}
\end{align}
for all $x\in \Omega$ and  $v\in\R^n$. It turns out that $\Delta_d$ is a convex Finsler metric. Moreover, it can be proved that for every Lipschitz continuous curve $\gamma:[a,b]\to \Omega$, we have
\begin{align*}
\mathcal{L}_d(\gamma)=\int_a^b\Delta_d(\gamma,\gamma')dt.
\end{align*}
These facts can be found for instance in~\cite[Section 1.1]{gpp06}.

By~\cite[Proposition 1.6]{gpp06}, for an admissible Finsler structure $F$, one always has
\begin{align}\label{eq:nocoincideness of metric derivative and differential}
\Delta_{d_c^*}(x,v)\leq F^*(x,v).
\end{align}
However, for general Finsler structures, the strict inequality above can hold;  see e.g.~\cite[Example 5.1]{dp95}.

\section{Weak coincidence of differential structure and distance structure}

The following proposition is crucial in the proof of the existence part of Theorem \ref{thm:  Existence and Uniqueness} in Section \ref{sec: Existence and Uniqueness}. For its proper formulation, we recall that the pointwise Lipschitz constant function $\Lip_d u$ of a Borel function $u:\Omega\to \bR$ with respect to a distance $d$ is defined as
\begin{align*}
\Lip_{d}u(x):=\limsup_{y\to x}\frac{|u(y)-u(x)|}{d(x,y)}.
\end{align*}

\begin{proposition}\label{prop:essential differential coincide with distance}
Let $F$ be an admissible Finsler structure on $\Omega$. Then for every open set $V\ssub \Omega$ and every  function $u\in \Lip_{\loc}(\Omega)$, we have
\begin{equation*}
\underset{x\in V}{\esup} F(x,\nabla u(x))= \underset{x\in V}{\esup} \Lip_{\delta_F}u(x)=\sup_{x\in V}\Lip_{\delta_F}u(x).
\end{equation*}
\end{proposition}

\begin{proof}
Let $V\ssub \Omega$ and  $u\in \Lip_{\loc}(\Omega)$ be an arbitrary Lipschitz continuous function. We first show that
\begin{equation}\label{eq: 3.1.1}
 \underset{x\in V}{ \esup} F(x,\nabla u(x))\geq \sup_{x\in V}\Lip_{\delta_F}u(x).
\end{equation}
Since both sides of (\ref{eq: 3.1.1}) are positively 1-homogeneous with respect to $u$, 
we only need to show that if $\esup_{x\in V}F(x,\nabla u(x))\leq 1$, then $\sup_{x\in V}\Lip_{\delta_F}u(x)\leq 1$.

By the definition of $\delta_F$, if $\esup_{x\in V}F(x,\nabla u(x))\leq 1,$
then $|u(x)-u(y)|\leq \delta_F(x,y)$ for all $x,y\in V$, which implies that
\begin{align*}
 \sup_{x\in V}\Lip_{\delta_F}u(x)\leq 1.
\end{align*}
Thus, we obtain (\ref{eq: 3.1.1}).

We next show that
\begin{equation}\label{eq: 3.1.2}
 \underset{x\in V}{ \esup} F(x,\nabla u(x))\leq  \underset{x\in V}{ \esup} \Lip_{\delta_F}u(x).
\end{equation}
Since both sides of  (\ref{eq: 3.1.2}) are positively
1-homogeneous with respect to $u$, we only need to show that for
a.e. $x\in V$, if $\Lip_{\delta_F}u(x)\leq 1$, then $F(x,\nabla
u(x))\leq 1$.

Note that by Lemma~\ref{lemma:key lemma},
$\Lip_{\delta_F}u(x)=\Lip_{d_c^*}u(x)$ for a.e. $x\in \Omega$. Fix
such an $x$. For each $v\in \bR^n$, we have
\begin{align*}
\langle \nabla u(x),v \rangle&=\lim_{t\to 0}\frac{u(x+tv)-u(x)}{t}\\
&\leq \limsup_{t\to 0}\frac{d_c^*(x+tv,x)}{t}\limsup_{t\to 0}\frac{u(x+tv)-u(x)}{d_c^*(x+tv,x)}\\
&\leq \Delta_{d_c^*}(x,v)\Lip_{d_c^*}u(x)\\
&\leq \Delta_{d_c^*}(x,v)\leq F^*(x,v),
\end{align*}
where the last inequality follows from~\eqref{eq:nocoincideness of metric derivative and differential}. Therefore, by Proposition \ref{prop:basic prop of dual} we have
\begin{align*}
F(x,\nabla u(x))&=F^{**}(x,\nabla u(x))
=\sup_{v\neq 0}\left\langle\nabla u(x),\frac{v}{F^*(x,v)}\right\rangle\leq 1.
\end{align*} This proves (\ref{eq: 3.1.2}).

Combining (\ref{eq: 3.1.1}) and (\ref{eq: 3.1.2}) gives us that
\begin{align*}
 \underset{x\in V}{ \esup} F(x,\nabla u(x))&\leq  \underset{x\in V}{ \esup}\Lip_{\delta_F}u(x)\\
&\leq \sup_{x\in V}\Lip_{\delta_F}u(x)\\
&\leq  \underset{x\in V}{ \esup} F(x,\nabla u(x)),
\end{align*}
which implies that  all the inequalities above are actually equalities. The proof of Proposition~\ref{prop:essential differential coincide with distance} is complete.
\end{proof}

\section{Existence and Uniqueness }\label{sec: Existence and Uniqueness}

Let $F$ be an admissible Finsler structure on $\Om$ and
$U\subset\subset\Omega$. In this section we prove the following
theorem.

\begin{theorem} \label{thm:  Existence and Uniqueness}
{\rm(i)} For every
$f\in\lip(\partial U)$, there exists a unique absolutely minimizing
Lipschitz extension on $U$ with respect to $F$.

{\rm(ii)} The absolute minimizer is completely determined by the intrinsic
distance in the following sense: let $\delta_{F}$ and $\delta_{\tilde{F}}$
be the intrinsic distance associated with the admissible Finsler structures $F$ and $\tilde{F}$, respectively.
If for almost all $x\in U$ there holds
\begin{equation}
\lim_{x\ne y\to x}\frac{\delta_{F}(x,\, y)}{\delta_{\tilde{F}}(x,\, y)}=1,\label{e4.xx1}
\end{equation}
 then $u$ is an absolute minimizer on $U$
for $F$ if and only if $u$ is an absolute minimizer on $U$ for
$\tilde{F}$.
\end{theorem}

Note that  Theorem \ref{thm:main intro} is the first part of Theorem \ref{thm:  Existence and Uniqueness}.  
The proof of Theorem \ref{thm:  Existence and
Uniqueness} is long and thus is divided into several lemmas. In
the following,  we first prove the existence part, and then the
uniqueness part  of Theorem \ref{thm: Existence and Uniqueness}(i). 
In the end of this section, we give a complete proof of
Theorem \ref{thm: Existence and Uniqueness}.

\subsection{Proof of existence} The following lemma is an analogy of \cite[Lemma 7]{ksz-2014}, which characterizes absolute minimizers via intrinsic distances.

\begin{lemma}\label{l4.5} Let $u\in\lip_{\loc}(U)$. Then $u$ is an absolute
minimizer on $U$ if and only if for each bounded open subset $V\subset\subset U$
and all $v\in\lip_{\loc}(V)\cap C(\overline{V})$ with $u|_{\partial V}=v|_{\partial V}$,
one (or both) of the following holds:

{\rm(i)} \noindent  $\esup_{x\in V}\lip_{\delta_{F}}u(x)\le{\esup}_{x\in V}\lip_{\delta_{F}}v(x);$

{\rm(ii)} \noindent  $\sup_{x\in V}\lip_{\delta_{F}}u(x)\le{\sup}_{x\in V}\lip_{\delta_{F}}v(x).$
\end{lemma}
\begin{proof}
In view of Proposition \ref{prop:essential differential coincide with distance}, Lemma \ref{l4.5}
 is no more than a restatement of the definition of absolute minimizers.\end{proof}

Notice that our concept of absolutely minimizing Lipschitz extensions
defined in Section 1
corresponds to the strongly absolutely minimizing Lipschitz extension in \cite{jn}.
Applying Lemma~\ref{l4.5} and \cite[Theorem 3.1]{jn}, we have
the following existence result.

\begin{lemma}\label{l4.8}
For every $f\in \lip(\partial U)$, there exists an absolutely minimizing Lipschitz extension of $f$ on $U$.
\end{lemma}

\subsection{Proof of uniqueness}
We point out here that the existence  and   the uniqueness   of
absolutely minimizing Lipschitz extensions in domains in a length
space have already been proven in~\cite{pssw} via a probabilistic
approach called the Tug-of-War. Here to prove the uniqueness
result,  we derive the following comparison principle, by applying
the strategy developed by Armstrong and Smart~\cite{as}.

\begin{lemma}\label{l4.9}
Let $u,\,v\in \lip_{\loc}(U)\cap C(\overline U)$ be absolute minimizers on $U$. Then
$$\max_{x\in \overline U}[u(x)-v(x)]= \max_{x\in \partial U}[u(x)-v(x)].$$
\end{lemma}

Before going into the proof of Lemma \ref{l4.9}, let us first
recall the definition of the comparison with cones introduced by
Crandall et al.~\cite{ceg}. A function $u\in C(U)$ is said to
satisfy the {\it property of comparison with cones}
 if
for all subsets $V\subset\subset U$ and  all $a\ge0$, $b\in\rr$ and
$x_0\in\rn\setminus V$, we have

(\textrm{I}) $\max_{x\in \partial V}[u(x)-C_{b,a,x_0}(x)]\le0$ implies
$\max_{x\in V}[u(x)-C_{b,a,x_0}(x)]\le0;$

(\textrm{II}) $\max_{x\in \partial V}[u(x)-C_{b,-a,x_0}(x)]\ge0$ implies
$\max_{x\in V}[u(x)-C_{b,-a,x_0}(x)]\ge0,$

\noindent
where the cone function $C_{b,a,x_0}$ is defined as
$$C_{b,a,x_0}(x):=b+a\,\delta_F(x,\,x_0).$$
It is known that an absolute minimizer satisfies the comparison property
with cones;
see \cite{ceg} for Euclidean case and \cite{acj,jn,gwy,cp,dmv} for the
setting of metric spaces that are length spaces.

The following is a list of  equivalent characterizations for absolute minimizers.

\begin{lemma}\label{l4.53}
The following statements are mutually equivalent:

\noindent{\rm(i)}  $u$ is an absolute  minimizer on $U$.

\noindent{\rm(ii)}  for all open sets $V\subset\subset U$,
$\lip_{\delta_F}(u,\,V)=\lip_{\delta_F}(u,\,\partial V)$.

\noindent{\rm(iii)}  $u$ satisfies the property of comparison with cones.
\end{lemma}

\begin{proof}
(i)$\Rightarrow$(ii). It is a consequence of Lemma~\ref{l4.5}.
Indeed, notice that for every pair $x,\,y\in\partial V$ with $x\ne
y$, by the continuity of  $\delta_F$ we can find $x_n,\,y_n\in V$
such that $x_n\to x$ and $y_n\to y$.  By the continuity of $u$, we
have
  $$\frac{|u(x_n)-u(y_n)|}{\delta_F(x_n,\,y_n)}\to\frac{|u(x)-u(y)|}{\delta_F(x,\,y)}\quad\text{as}\ n\to\infty.$$
Hence $\lip_{\delta_F}(u,\,V)\ge \lip_{\delta_F}(u,\,\partial V)$
holds. Thus, it suffices to prove that $\lip_{\delta_F}(u,\,V)\le
\lip_{\delta_F}(u,\,\partial V)$.

For $x\in\rn$, let $$w(x):=\sup_{z\in\partial V}[u(z)+\lip_{\delta_F}(u,\,\partial V) \delta_F(x,\,z)].$$
 Then $\lip_{\delta_F}(w,\,\rn)=\lip_{\delta_F}(u,\,\partial V)$
and  $w=u$ on $\partial V$.
Applying Lemma \ref{l4.5}, we have $$\sup_{x\in V}\lip_{\delta_F} u(x)\le \sup_{x\in V}\lip_{\delta_F} w(x)\le
 \lip_{\delta_F}(u,\,\partial V).$$
Notice  that $(U,\,\delta_F)$ is a geodesic space. Indeed, since
$U\ssub\Omega$ is open bounded, our assumption on $F$ implies that
there exist positive constants $\alpha:=\alpha(U)$ and $\beta:=\beta(U)$ such
that
\begin{align*}
\alpha |v|\leq F(x,v)\leq \beta |v|
\end{align*}
for all $x\in U$ and $v\in \bR^n$. Combining this fact together with~\cite[Theorem 3.9]{gpp06} yields that $(U,\,\delta_F)$ is a geodesic space. Thus, given a pair of points  $x,\,y\in U$, we can select  a $\delta_F$-geodesic curve $\gz$ joining $x$ and $y$.

If $\gz\subset V$, then
\begin{eqnarray*}|u(x)-u(y)|&\le& \int_\gz\lip_{\delta_F}\, u(z)\,{\rm \text{d}}s\\
&\le &\delta_F(x,\,y)\,\sup_{x\in V}\lip_{\delta_F} u(x)\\
    &\le & \delta_F(x,y)\, \lip_{\delta_F}(u,\,\partial V). \end{eqnarray*}
Here ${\rm \text{d}}s$ denotes arc-length integral on $\gz$ with respect to the metric
$\delta_F$.
If $\gamma\not\subset V$, denote by $\hat x$, $\hat y\in \gz\cap\partial V$
points that have shortest distance to $x$ and $y$,
respectively.
Then
\begin{eqnarray*}
|u(x)-u(y)|&\le &|u(x)-u(\hat x)|+|u(\hat x)-u(\hat y)|+|u(\hat y)-u(y)|\\
 &\le &[\delta_F( x,\,\hat x)+\delta_F( y,\,\hat y)]\sup_{x\in V}\lip_{\delta_F} u(x)  +\delta_F( \hat x,\,\hat y)\lip_{\delta_F}(u,\,\partial V)\\
 &\le& \delta_F(  x,\, y)\lip_{\delta_F}(u,\,\partial V).
\end{eqnarray*}
Thus, in both cases, we have the estimate
\[
   \frac{|u(x)-u(y)|}{\delta_F(x,y)}\le \lip_{\delta_F}(u,\,\partial V),
\]
which implies  that $\lip_{\delta_F}(u,\,V)\le \lip_{\delta_F}(u,\,\partial V)$.

(ii)$\Rightarrow$(iii). We prove~(\textrm{I}) by a contradiction
argument. The proof of~(\textrm{II}) is similar (and left to the
interested reader). Let $u$ be an absolute minimizer and assume
that
$$\max_{x\in \partial V}[u(x)-C_{b,a,x_0}(x)]\le0.$$ Suppose that
(\textrm{I}) fails, that is,  $\max_{x\in
V}[u(x)-C_{b,a,x_0}(x)]>0$. Denote by $W$ the open set of all
$x\in V$ such that $u(x)>C_{b,a,x_0}(x)$. By assumption, $W$ is
not empty. Moreover, we have $u= C_{b,a,x_0}$ on $\partial
W$. Since $W\subset V\subset\subset U$, by assumption (ii) we have
$$\lip_{\delta_F}(u,\overline W)=\lip_{\delta_F}(u,W)=\lip_{\delta_F}(u,\partial W)=a.$$
For $x\in W$, let
$\gz$ be a $\delta_F$-geodesic curve joining $x$ and $x_0$, and
take $z\in\partial W\cap \gz$ be a closest point to $x$.
Then \begin{eqnarray*}
u(x)-u(z)&>&C_{b,a,x_0}(x)-C_{b,a,x_0}(z)\\
&=& a\delta_F(x_0,\,x)-a{\delta_F}(x_0,\,z)\\
&=&a\delta_F(x,\,z),
\end{eqnarray*}
which implies that $\lip_{\delta_F}(u,\,W)>a$. We reach a contradiction. So $W$ must be empty.
Therefore (ii) implies (iii).

 (iii)$\Rightarrow$(i).
We only need to notice that, with the help of Proposition \ref{prop:essential differential coincide with distance},
 the argument provided by the proof of \cite[Proposition 5.8]{jn} still works
 here, without the additional weak Fubini property required in \cite{jn};
see also \cite{acj}.    This completes the proof of Lemma \ref{l4.53}.
\end{proof}

 With the aid of Lemmas~\ref{l4.5} and \ref{l4.53},
Lemma~\ref{l4.9} will be proved by following the procedure from~\cite{as}.
Since the proof in \cite{as} is for the  case $F(x, \cdot):=|\cdot|$,
we write down the details below for the reader's convenience. We need some notation.
For all $r>0$, let $$ U_r:=\{z\in U:\ \overline{B_{\delta_F}(z,\,r)}\subset U\}.$$ For $x\in U_r$,
we  let
$$S_r^+u(x):= \frac{u^r(x)-u(x)}{r}\quad\text{and}\quad S_r^-u(x):=\frac{u(x)-u_r(x)}{r},$$
where $u^r(x):= \sup_{\delta_F(z,\,x)\le r} u(z)$ and $u_r(x):= \inf_{\delta_F(z,\,x)\le r} u(z)$.
\begin{proof}[Proof of Lemma \ref{l4.9}.]
First we claim that for $x\in U_{2r}$, we have
\begin{equation}\label{e4.x5}
S_r^-u^r(x)-S_r^+u^r(x)\le0\le S_r^-v_r(x)-S_r^+v_r(x).
\end{equation}
Indeed, let $y\in \overline {B_{\delta_F}(x,\,r)}$ and $z\in \overline {B_{\delta_F}(x,\,2r)}$
such that $u^r(x)=u(y)$ and $(u^{r})^r(x)=u^{2r}(x)=u(z)$.  Observe
 that $(u^r)_r(x)\ge u(x)$. Then we have
  \begin{equation}
    \begin{aligned}
  S_r^-u^r(x)-S_r^+u^r(x)&=\frac1{r}\big[2u^r(x)-(u^r)^r(x)-(u^r)_r(x)\big]\\
&\le
 \frac 1{r}\big[2u (y)-u(z)-u(x)\big].\end{aligned} \label{eq: 4.star-2}
  \end{equation}
Note that for $w\in\boz$ such that $\delta_F(x,\,w)=2r$, we have
$$u(w)\le u (z)=u(x)+ [u (z)-u(x)]= u(x)+\frac{[ u (z) -u(x)]}{2r}\delta_F(w,\,x).$$
Thus, the comparison with cones property of $u$ implies that the inequality
\[
   u(w)\le u(x)+\frac{[ u (z) -u(x)]}{2r}\delta_F(w,\,x)
\]
holds for all
$w\in\boz$ with $\delta_F(x,\,w)\le 2r$. In particular,  taking $w=y$ and noticing  that $\delta_F(y,\,x)\le r$, we obtain
\begin{align*}u(y)&\le  u(x)+\frac{[ u (z) -u(x)]}{2r}\delta_F(y,\,x)\\
&\le   u(x)+\frac12{[ u (z) -u(x)]}\\
&=\frac12{[ u (z)+u(x)]},
\end{align*}
which, together with (\ref{eq: 4.star-2}), implies the first inequality of \eqref{e4.x5}. The second inequality of \eqref{e4.x5} follows similarly.

Next we claim that~\eqref{e4.x5}  gives us that
\begin{equation}\label{e4.xx5}
    \sup_{x\in U_r}[u^r(x)-v_r(x)]= \sup_{x\in U_r\setminus U_{2r}}[u^r(x)-v_r(x)]
\end{equation}
for all $r>0$. Once we prove that~\eqref{e4.xx5} holds, then  Lemma~\ref{l4.9} follows by  letting $r\to0$ in~\eqref{e4.xx5}. Thus,  we only need to prove~\eqref{e4.xx5}.

Suppose, on the contrary,  that \eqref{e4.xx5} dose not hold.
Then there exists  some $r>0$ for which
\begin{equation}\label{eq: assume e4.xx5  not true}
\sup_{x\in U_r}[u^r(x)-v_r(x)] > \sup_{x\in U_r\setminus U_{2r}}[u^r(x)-v_r(x)].
\end{equation}
By the continuity of $u^r -v_r $, there must exist some $y\in \overline {U}_{r}$ such that
$$u^r(y)-v_r(y)=\sup_{x\in U_{ r}}[u^r(x)-v_r(x)].$$ Note that (\ref{eq: assume e4.xx5  not true}) implies that $y\in U_{2r}$.
Denote by $E$ the set of all such $y$ and let $$K:=\left\{x\in E:
\ u^r(x)=\max_{z\in E} u^r(z)\right\}.$$ Then $K$ is a closed subset of  $U_{2r}$  by the continuity of $u^r$ again.
Choose $x_0\in\partial K$. Since $x_0\in E$, for every $x\in U_r$ we have
\[
   u^r(x_0)-v_r(x_0)\ge u^r(x )-v_r(x ).
\]
Since $x_0\in U_{2r}$, we have  $ B_{\delta_F}(x_0,r)\subset U_r$. Thus, for every $x\in B_{\delta_F}(x_0,r)$, we deduce from above inequality  that
\[
   u^r(x_0)-v_r(x_0)\ge \inf_{z\in B_{\delta_F}(x_0,r)} u^r(z)-v_r(x)=(u^r)_r(x_0)-v_r(x).
\]
That is,
\[u^r(x_0)-(u^r)_r(x_0) \geq  v_r(x_0) -v_r(x).\]
Divide by $r$ on each side of above equality and then take the infimum over $x\in B_{\delta_F}(x_0,r)$. We obtain
\begin{equation}\label{eq: 4.star-1}
  S_r^-u^r(x_0) \ge S_r^-v_r(x_0).
\end{equation}
Now we  have two cases.

{\it Case 1:} $S^+_ru^r(x_0)=0$.

{\it Case 2:} $S^+_ru^r(x_0)>0$.

Consider {\it Case 1}. In this case,
 \eqref{e4.x5} yields that $S_r^-u^r(x_0)\le0$.
Hence  we derive that  $S_r^-u^r(x_0)=0$ holds,
which, together with (\ref{eq: 4.star-1}), implies that $S_r^-v_r(x_0)=0$. By \eqref{e4.x5} again, we have
$S_r^+v_r(x_0)\le0$ and hence $S_r^+v_r(x_0)=0$. So we obtain
$u^r\equiv u^r(x_0)$ and $v_r\equiv v_r(x_0)$ hold on
$  B_{\delta_F}(x_0,\,r)$. This contradicts to the fact that $x_0\in\partial K$.

It remains to consider {\it Case 2}.
Choose $z\in\overline B_{\delta_F}(x_0,\,r)$ such that
$$0<rS^+_ru^r(x_0)=u^r(z)-u^r(x_0).$$ Since $u^r(z)>u^r(x_0)$ and $x_0\in K$, it follows that
$z\notin E$. Note that $z\in U_r $ since $x_0\in U_{2r}$.
Therefore the fact that  $x_0\in E$ yields   $$u^r(x_0)-v_r(x_0)> u^r(z)-v_r(z).$$
 That is, we have   $$v_r(z)-v_r(x_0)>u^r(z)-u^r(x_0).$$
  Hence we derive   that
$$rS^+_rv_r(x_0)\ge v_r(z)-v_r(x_0)>u^r(z)-u^r(x_0)=rS^+_ru^r(x_0),$$
which, together with (\ref{eq: 4.star-1}), implies that
$$S^+_rv_r(x_0)-S_r^-v_r(x_0)> S^+_ru^r(x_0)-S_r^-u^r(x_0).$$
We obtain a contradiction to~\eqref{e4.x5}.

Since both cases above do not hold, we conclude that \eqref{eq: assume e4.xx5  not true} is not true for any $r>0$. That is,   \eqref{e4.xx5} holds. The proof of Lemma \ref{l4.9} is complete.
\end{proof}

Now we are in a position to prove Theorem \ref{thm:  Existence and Uniqueness}.
\begin{proof}[Proof of Theorem \ref{thm:  Existence and Uniqueness}]
Theorem \ref{thm:  Existence and Uniqueness}(i) follows from  Lemmas \ref{l4.8} and \ref{l4.9}.
Theorem \ref{thm:  Existence and Uniqueness}(ii) follows from Lemma \ref{l4.5} 
with the observation that under the assumption \eqref{e4.xx1},
$$\lip_{\delta_F}u=\lip_{\delta_{ \wz F}}u$$ 
almost everywhere for every $u\in \lip_\loc(\rn)$. The proof of Theorem \ref{thm:  Existence and Uniqueness} is complete.
\end{proof}

\section*{Appendix: an alternative proof of Lemma~\ref{lemma:key lemma}}

In the appendix, we give an alternative proof of Lemma~\ref{lemma:key lemma} based on an approximation argument similar to the proof of~\cite[Proposition 3.1]{g14}. The proof is based on a personal communication with Professor Andrea Davini. In particular, he draws our attention on the useful  reference~\cite{dp07} and carefully explains its relation with~\cite{d05}.  We would like to express our gratitude here for his kind help.

\begin{proof}[Proof of Lemma~\ref{lemma:key lemma}]
    The proof is similar to that of  \cite[Proposition 3.1 ]{g14}. The inequality $\delta_F(x,y)\leq d_c^*(x,y)$ follows directly from definition. Indeed, for each Lipschitz function $u$ with $\|F(x,\nabla u(x))\|_\infty\leq 1$, each $x,y\in \Omega$, for each Lipschitz curve $\gamma$ joining $x$ and $y$ that is transversal to the zero measure set $N:=\{x\in \Omega:F(x,\nabla u(x))> 1\}$,
    \begin{align*}
        u(x)-u(y)&=\int_\gamma \langle\nabla u(\gamma(t)),\gamma'(t)\rangle dt\\
        &\leq \int_\gamma F^*(\gamma(t),\gamma'(t))dt=\mathcal{L}_{d_c^*}(\gamma),
    \end{align*}
    where $\mathcal{L}_{d_c^*}$ denotes the length of the curve $\gamma$ with respect to the metric $d_c^*$. Taking infimum over all admissible curves on the right-hand side and then supermum over all admissible functions over the left-hand side, we obtain
    \begin{align*}
        \delta_F(x,y)\leq d_c^*(x,y).
    \end{align*}
    In particular,
    \begin{align*}
        \limsup_{y\to x}\frac{\delta_F(x,y)}{d_c^*(x,y)}\leq 1.
    \end{align*}

    So we are left to prove that
    \begin{align}\label{eq:liminf ineq}
        \liminf_{y\to x}\frac{\delta_F(x,y)}{d_c^*(x,y)}\geq 1.
    \end{align}
    When $F$ is continuous,  ~\eqref{eq:liminf ineq} holds by~\cite[Proposition 3.1]{g14}. In the general case when $F$ is only Borel measurable, we can use an approximation argument as follows. Since~\eqref{eq:liminf ineq} is at infinitesimal scale, we can assume that $F$ is uniform elliptic with absolute positive constants $\alpha$ and $\beta$. That is, $\alpha |v|\leq F(x,v)\leq \beta |v|$
  for all $x\in \Omega$ and $v\in \bR^n$. Then, by~\cite[Theorem 4.1]{d05} or~\cite[Section 2]{dp07}, there exists a sequence $\{F_n\}_{n\in \mathbb{N}}$ of continuous Finsler structures such that
    \begin{align*}
        d_c^{*n}\to d_c^*\quad \text{ and }\quad \limsup_{n\to \infty}\delta_{F_n}\leq \delta_F
    \end{align*}
    with respect to the uniform convergence of distances on $\Omega\times \Omega$,  where $\delta_{F_n}$ is the  distance induced by $F_n$ in the same way as that of $\delta_{F}$,  and $d_c^{*n}:=d_c^{F^*_n}$ is the distance induced by the dual of the Finsler structure $F_n$ for all $n$.

    Now, for any  $\varepsilon>0$, there exists a number $N_0>1$ such that for  $n\geq N_0$, we have
    \begin{align*}
        \frac{\delta_F(x,y)}{d_c^*(x,y)}\geq (1-\varepsilon)\frac{\delta_{F_n}(x,y)}{d_c^{*n}(x,y)}.
    \end{align*}
    On the other hand, since $F_n$ is continuous, by~\cite[Proposition 3.1]{g14}, we have
    \begin{align*}
        \liminf_{y\to x}\frac{\delta_{F_n}(x,y)}{d_c^{*n}(x,y)}\geq 1.
    \end{align*}
    Thus, we deduce that
    \begin{align*}
        \liminf_{y\to x}\frac{\delta_{F}(x,y)}{d_c^{*}(x,y)}\geq \liminf_{y\to x}(1-\varepsilon)\frac{\delta_{F_n}(x,y)}{d_c^{*n}(x,y)}\geq 1-\varepsilon.
    \end{align*}
    Sending $\varepsilon\to 0$ yields (\ref{eq:liminf ineq}). The proof of Lemma \ref{lemma:key lemma} is complete.
\end{proof}

\textbf{Acknowledgements}

Part of this research was done when C.Y.Guo was visiting at School of Mathematical Sciences, Beijing Normal University and Department of Mathematics, Beihang University in China during the period 11 June --15 July 2015. He wishes to thank both of them for their great hospitality.


\begin{thebibliography}{99}
%
%


\bibitem{as}
 Armstrong, S.N., Smart, C.K.:  An easy proof of Jensen's theorem on the
uniqueness of infinity harmonic functions. {\it Calc. Var. Partial Differential
Equations} {\bf 37}, 381-384  (2010)


\bibitem{a1}
Aronsson, G.:
Minimization problems for the functional $\sup_x F(x, f(x), f' (x))$.
{\it Ark. Mat.} {\bf 6}, 33-53   (1965)


\bibitem{a2}
Aronsson, G.:  Minimization problems for the functional $\sup_x F(x, f(x), f' (x))$. II.
{\it Ark. Mat.} {\bf 6},  409-431 (1966)


\bibitem{a3}
Aronsson, G.:   Extension of functions satisfying Lipschitz conditions.
{\it Ark. Mat.} {\bf 6},  551-561 (1967)


\bibitem{a4}
Aronsson, G.:   Minimization problems for the functional $\sup_x F(x, f(x), f' (x))$. III.
{\it Ark. Mat.} {\bf 7},
 509-512 (1969)





 \bibitem{acj}
Aronsson, G.,  Crandall, M.G., Juutinen, P.:   A tour of the theory of
absolutely minimizing functions. {\it Bull.
Amer. Math. Soc. (N. S.)} {\bf  41}  439-505 (2004)

\bibitem{bb01}
 Barles, G., Busca, J.: Existence and comparison results for fully nonlinear degenerate elliptic equations without zeroth-order term. Comm. Partial Differential Equations \textbf{26}, no. 11-12, 2323-2337 (2001)

\bibitem{cp}
Champion, T., De Pascale, L.: Principle of comparison with distance functions
for absolute minimzer. {\it
J. Convex Anal.} {\bf 14},   515-541 (2007)


\bibitem{cpp}
Champion, T., De Pascale,  L.,  Prinari, F.: $\Gamma$-Convergence and absolute
minimizers for supremal
functionals. {\it ESAIM Control Optim. Calc. Var.} {\bf 10},   14-27 (2004)


\bibitem{ceg}
 Crandall, M.G.,  Evans,  L.C.,  Gariepy, R.F.:  Optimal Lipschitz extensions
and the infinity Laplacian. {\it Calc.
Var. Partial Differential Equations}  {\bf 13}  123-139  (2001)

\bibitem{cgw07}
 Crandall, M.G., Gunnarsson, G., Wang, P.: Uniqueness of $\infty$-harmonic functions and the eikonal equation. Comm. Partial Differential Equations \textbf{32}, no. 10-12, 1587-1615 (2007)

\bibitem{d05}
Davini, A.: Smooth approximation of weak Finsler metrics, \textit{Differential Integral Equations} \textbf{18}, no. 5, 509-530 (2005)

\bibitem{dp07}
Davini, A., Ponsiglione, M.: Homogenization of two-phase metrics and applications, \textit{J. Anal. Math.} \textbf{103}, 157-196 (2007)

\bibitem{dp93}
De Cecco, G., Palmieri, G.: Intrinsic distance on a LIP Finslerian manifold, (Italian) \textit{Rend. Accad. Naz. Sci. XL Mem. Mat.} (5) \textbf{17}, 129-151 (1993)

\bibitem{dp95}
De Cecco, G., Palmieri, G.:  LIP manifolds: from metric to Finslerian structure, \textit{Math. Z.} \textbf{218}, no. 2, 223-237 (1995)

\bibitem{dmv}
Dragoni, F., Manfredi, J.J., Vittone,  D.:
Weak Fubini property and infinity
harmonic functions in Riemannian and
sub-Riemannian manifolds. {\it Trans. Amer. Math. Soc.}
{\bf 365},    837-859 (2013)


\bibitem{es}
Evans, L.C., Savin, O.: $C^{1,\az}$ regularity for infinity harmonic
functions in two dimensions.
{\it Calc. Var. Partial Differential Equations} {\bf 32}, 325-347  (2008)


\bibitem{es12}
Evans, L.C.,  Smart, C.K.:
Everywhere differentiability of infinity harmonic functions.
 {\it Calc. Var. Partial Differential Equations}  {\bf 42},  289-299 (2011)




 \bibitem{gpp} Garroni, A.,  Ponsiglione, M.,
Prinari, F.:   From 1-homogeneous supremum functional to difference
quotients: relaxation and $\Gamma$-convergence.
 {\it Calc. Var. Partial Differential Equations}
 {\bf 27}, 397-420  (2006)



 \bibitem{gwy}
 Gariepy,  R.,   Wang, C., Yu, Y.:
 Generalized cone comparison principle for viscosity solutions of the
 Aronsson equation and absolute minimizers.
 {\it Comm. Partial Differential Equations} {\bf 31},  1027-1046  (2006)


\bibitem{gpp06}
Garroni, A., Ponsiglione, M., Prinari, F.: From 1-homogeneous supremal functionals to difference quotients: relaxation and $\Gamma$-convergence, \textit{Calc. Var. Partial Differential Equations} \textbf{27}, no. 4, 397-420 (2006)

\bibitem{gnp01}
 Garroni, A., Nesi, V., Ponsiglione, M: Dielectric breakdown: optimal bounds, R. Soc. Lond. Proc. Ser. A Math. Phys. Eng. Sci. \textbf{457}, no. 2014, 2317-2335 (2001)
 
\bibitem{g14}
Guo, C.Y.: Intrinsic geometry and analysis of Finsler structures, \textit{submitted} 2014.

\bibitem{hil05}
 Hassani, Riad, Ionescu, Ioan R., Lachand-Robert, T.: Shape optimization and supremal minimization approaches in landslides modeling, Appl. Math. Optim. \textbf{52}, no. 3, 349-364 (2005)


 \bibitem{j93}
Jensen, R.:   Uniqueness of Lipschitz extensions: minimizing the sup norm of
the gradient. {\it Arch. Ration.
Mech. Anal.} {\bf 123},   51-74 (1993)


 \bibitem{ju}
Juutinen, P.:  Absolutely minimizing Lipschitz extensions on a metric space.
{\it Ann. Acad. Sci. Fenn. Math.} {\bf 27},  57-67   (2002)


 \bibitem{jn}
Juutinen, P.,    Shanmugalingam, N.:  Equivalence of AMLE, strong AMLE, and
comparison with cones in metric measure space.
{\it Math. Nachr.} {\bf 279},  1083-1098  (2006)


 \bibitem{kz}
Koskela,  P.,   Zhou, Y.: Geometry and analysis of Dirichlet forms.
{\it Adv. Math.}
{\bf 231},  2755-2801 (2012)


 \bibitem{ksz}
Koskela,  P.,  Shanmugalingam, N.,  Zhou, Y.: $L^\infty$-variational problem on
metric measure spaces. {\it Math. Res. Lett.} {\bf 19}, 1263-1275   (2012)

\bibitem{ksz-2014}
Koskela,  P.,  Shanmugalingam, N.,  Zhou, Y.: Intrinsic geometry and analysis of diffusion processes and $L^{\infty}$-variational problems. {\it Arch. Ration. Mech. Anal.} {\bf 214}, 99-142   (2014)




\bibitem{pssw}
Peres, Y., Schramm, O., Sheffield, S.,    Wilson, D.B.:  Tug-of-war and the
infinity Laplacian.
{\it J. Amer. Math. Soc.} {\bf 22},   167-210 (2009)



\bibitem{s05}
Savin, O.: $C^1$ regularity for infinity harmonic functions in two dimensions.
{\it Arch. Ration. Mech. Anal.}  {\bf 176},  351-361  (2005)

\bibitem{swz14}
Siljander, J., Wang, C.Y., Zhou, Y.: Everywhere differentiability of viscosity solutions to a class of Aronsson's equations, \textit{preprint} 2014.


\bibitem{wy}
Wang, C.,    Yu, Y.:   $C^1$-regularity of the Aronsson equation in
${\mathbf R}^2$.
 {\it Ann. Inst. H. Poincar\'e Anal. Non Lin\'eaire} {\bf 25},  659-678   (2008)
\end{thebibliography}
\end{document}